\documentclass{amsart}
\usepackage{amsmath}
\usepackage{amsfonts}
\usepackage{amssymb}
\usepackage{epsfig}
\newtheorem{theorem}{Theorem}[section]
\newtheorem{theo}{Theorem}
\newtheorem{lm}[theorem]{Lemma}

\newtheorem{cor}[theorem]{Corollary}

\newtheorem{df}[theorem]{Definition}
\newtheorem{rem}[theorem]{Remark}
\newtheorem{pr}[theorem]{Proposition}

\begin{document}

\title[Central and supersymmetric elements]{Central elements in the distribution algebra of a general linear supergroup and supersymmetric elements}
\author{Franti\v sek Marko}
\address{Pennsylvania State University \\ 76 University Drive \\ Hazleton, PA 18202 \\ USA}
\email{fxm13@psu.edu}
\author{Alexandr N. Zubkov}
\address{Omsk State Technical University, Mira 11, 644050, Russia}
\email{a.zubkov@yahoo.com; alex.zubkov58@mail.ru}
\begin{abstract} In this paper we investigate the image of the center $Z$ of the distribution algebra $Dist(GL(m|n))$ of the general linear supergroup over a ground field of positive characteristic under the Harish-Chandra morphism $h:Z \to Dist(T)$ obtained by the restriction of the natural map $Dist(GL(m|n))\to Dist(T)$. We define supersymmetric elements 
in $Dist (T)$ and show that each image $h(c)$ for $c\in Z$ is supersymmetric. The central part of the paper is devoted to a description of a minimal set of generators of the algebra of supersymmetric elements over Frobenius kernels $T_r$.
\end{abstract}
\maketitle

\section*{Introduction}
We start by recalling a part of the generalization of the classical Harish-Chandra theory from algebraic groups to supergroups presented in \cite{chengwang}.
Let a ground field $K$ be of characteristic zero, $G=GL(m|n)$ be the general linear supergroup, $\mathfrak{g}=\mathfrak{gl}(m|n)$ be the general linear Lie superalgebra and 
$U=U(\mathfrak{g})$ be the enveloping superalgebra of $\mathfrak{g}$, and $Z$ be the center of $U$.
Let $\mathfrak{g}=\mathfrak{u}^-\oplus \mathfrak{h} \oplus \mathfrak{u}^+$ be a triangular decomposition. The PBW Theorem implies the decomposition 
$U=U(\mathfrak{h})\oplus (\mathfrak{u}^-U+U\mathfrak{u}^+)$ and the corresponding projection $\Phi: U\to U(\mathfrak{h})$, where $U(\mathfrak{h})$ is the enveloping algebra of the Cartan algebra $\mathfrak{h}$. The restriction of $\phi$ to $Z$ is an algebra homomorphism $h:Z\to U(\mathfrak{h})$. It turns out that the image $h(Z)$ is isomorphic to the algebra of supersymmetric polynomials in variables $(x_1, \ldots, x_m|y_1, \ldots, y_n)$. This property is later used to derive the linkage principle for $GL(m|n)$. For more details, please consult Section 2.2. Harish-Chandra homomorphism and linkage of the book \cite{chengwang}.

The Harish-Chandra homomorphism for $GL(m|n)$ over a field $K$ of positive characteristic $p$ was defined and studied in Section 8 of \cite{zubmar}. Instead of enveloping superalgebras 
$U(\mathfrak{g})$ and $U(\mathfrak{h})$ we considered the distribution algebras $Dist(G)$ and $Dist(T)$ of $G$ and its torus $T$. The natural map 
$Dist(G) \to Dist(T)$ restricted to the center 
$Z$ of $Dist(G)$ induces a superalgebra homomorphism $h:Z \to Dist(T)$. In Theorems 9.9 and 9.10 we have determined generators of $Z$ and of its image $h(Z)$ 
under the map $h$ for $GL(1|1)$.

The purpose of this paper is to investigate the image $h(Z)$ of the center $Z$ of $Dist(G)$ under the above map $h$ for a general $GL(m|n)$. We will define a subalgebra $SS$ of supersymmetric elements of $Dist(T)$, and show that $h(Z)$ is included in $SS$.
We expect that the equality $h(Z)=SS$ is valid for general $G$. We confirm this for $GL(1|1)$, the only case when the images $h(Z)$ are known. 
The main part of the paper is devoted to finding minimal generators of the algebras $SS$ for general $GL(m|n)$. 
Instead of $G$ and $T$, we will work with their Frobenius kernels $G_r$ and $T_r$. 

The structure of the paper is a follows. In Section 1, we derive properties of the distribution algebra $Dist(T)$ of a torus $T$ of $G$. In particular, we determine a complete set of primitive orthogonal idempotents $h_{a_1, \ldots, a_m|b_1, \ldots, y_n}$ of $Dist(T_r)$. In Section 2, we study elements $c$ of the center $Z$ of $Dist(G)$ and the properties of $h(c)$ under the morphism $h$. In Section 3, we use the properties derived in Section 2 to define the algebra $SS$ of supersymmetric elements in $Dist(T)$.
In Section 4, we determine specific types of supersymmetric elements. In Section 5, for $GL(m|1)$, we determine a basis and the dimension of $SS_r$. In Section 6, we define canonical elements and compute their dimension for every $GL(m|n)$. Finally, in Section 7, we establish the basis of $SS_r$ for every $GL(m|n)$.  


Throughout the paper, we assume that the characteristic of the ground field $K$ is $p>0$.

\section{The distribution algebra of a torus $T$}

We start by stating some elementary properties of the distribution algebra of a torus.

Let $T$ be a torus corresponding to diagonal matrices of size $m$.
The distribution algebra $Dist(T)$ of the torus $T$ has generating elements
$\binom{e_{ii}}{k}=\frac{e_{ii}(e_{ii}-1)\ldots (e_{ii}-k+1)}{k!}$ for $i=1, \ldots, m$ and $k>0$.
For simplicity, we will write $x_i=e_{ii}$ for $i=1, \ldots, m$.

The multiplication in $Dist(T)$ is such that $\binom{x_i}{a}$ and $\binom{x_j}{b}$ commute when $i\neq j$ while
\begin{equation}\binom{x_i}{a}.\binom{x_i}{b} = \sum_{i=0}^{\min\{a,b\}} \binom{a+b-i}{a-i}\binom{b}{b-i} \binom{x_i}{a+b-i}\end{equation}
for each $i=1, \ldots m$.
In particular, 
\[x_i\binom{x_i}{k}=(k+1)\binom{x_i}{k+1}+k\binom{x_i}{k}.\]

Denote by $T_r$ the $r$-th Frobenius kernel of $T$, and by $Dist(T_r)$ its distribution algebra.

If $0\leq a,b<q$, and $0\leq i\leq \min\{a,b\}$ is such that $a+b-i\geq q$, then there is a $p$-adic carry when $a-i$ is added to $b$. By Kummer's criteria, this means that
$\binom{a+b-i}{a-i}$ is divisible by $p$. Therefore, the $K$-span of the set $\mathcal{B}_r$ of monomials $\binom{x_1}{a_1}\ldots \binom{x_m}{a_m}$, where each $0\leq a_i<q$, is closed with respect to the multiplication. The set $\mathcal{B}_r$ is a basis of $Dist(T_r)$.

\subsection{$m=1$}

Denote by $G_m$ the one-dimensional torus. 

First, we consider the multiplicative structure of $Dist((G_m)_r)$ and write $x=x_1$.  

We will require the following combinatorial lemma. 
\begin{lm}\label{garrappa} Let $a>0$.
If $0\leq b<a$, then \[\sum_{0\leq r\leq a} (-1)^r \binom{a}{r}r^b =0.\]
If $b=a$, then \[\sum_{0\leq r\leq a} (-1)^r \binom{a}{r}r^a = (-1)^a a!\]
\end{lm}
\begin{proof}
The proof is given by induction.
The case $b=0$ or $a=1$ is clear.

Assume the statement is true for $a-1>0$ and consider $0<b\leq a$. Then 
\[\sum_{0\leq r\leq a} (-1)^r \binom{a}{r}r^b = \sum_{0<r\leq a} (-1)^r \binom{a}{r}r^b\] 
which equals \[a\sum_{0<r\leq a} (-1)^r \binom{a-1}{r-1}r^{b-1}\] because $\binom{a}{r}=\frac{a}{r}\binom{a-1}{r-1}$.
The last expression equals 
\[-a\sum_{0\leq r\leq a-1} (-1)^r \binom{a-1}{r}(r+1)^{b-1}=-a\sum_{0\leq j\leq b-1} \binom{b-1}{j} \sum_{0\leq r\leq a-1} (-1)^r \binom{a-1}{r}r^j\]
If $b<a$, then this sum vanishes by the inductive assumption.
If $b=a$, then it equals $(-a)(-1)^{a-1} (a-1)! = (-1)^a a!$.
\end{proof}

For $0\leq a<q$, denote 
\[X_a=\sum_{k=a}^{q-1} (-1)^{k-a}\binom{k}{a} \binom{x}{k}.\]

\begin{lm}\label{onedim}
We have $\binom{x}{i}X_i=X_i$ and $\binom{x}{j}X_i=0$ for $j>i$.
\end{lm}
\begin{proof}
We have 
\[\binom{x}{j}X_i = \sum_{i\leq k <q} (-1)^{k-i} \sum_{0\leq t \leq {\min\{j,k\}}} \binom{k}{i}\binom{j+k-t}{j-t}\binom{k}{t} \binom{x}{j+k-t}\]
which equals 
\[\sum_{j+s<q} \sum_{i\leq k\leq j+s} (-1)^{k-i} \binom{k}{i}\binom{j+s}{k}\binom{k}{s} \binom{x}{j+s}\]
where we set $s=k-t$. Since 
$\binom{k}{i}\binom{j+s}{k}=\binom{j+s}{i}\binom{j+s-i}{k-i}$, we rewrite the last sum as
\[\sum_{j+s<q} \binom{j+s}{i} \sum_{0\leq r\leq j-i+s} (-1)^r \binom{j-i+s}{r}\binom{i+r}{s},\]
where we set $r=k-i$.

If we split $\binom{r+i}{s}= \sum_{l=0}^s c_{l,i} r^l$ for appropriate coefficients $c_{l,i}$ and 
observe that $c_{s,i}=\frac{1}{s!}$, then Lemma \ref{garrappa} implies that 
\[\sum_{0\leq r\leq t+s} (-1)^r \binom{t+s}{r}\binom{i+r}{s}=0\] for $t>0$
while 
\[\sum_{0\leq r\leq s} (-1)^r \binom{s}{r}\binom{i+r}{s} = (-1)^s\]
and  $\binom{x}{i}X_i=X_i$, $\binom{x}{j}X_i=0$ for $j>i$ follow.
\end{proof}

\begin{lm}\label{torus1}
For each $r>0$, $Dist((G_m)_r)$ is a separable algebra.
Elements $X_a$ as above form a complete set of primitive orthogonal idempotents in $Dist((G_m)_r)$ and 
$\sum\limits_{0\leq a <q} X_a =1$.
\end{lm}
\begin{proof}
The proof follows from Lemma \ref{onedim} and  
\[\sum_{0\leq a <q} X_a =\sum_{0\leq a<q} \sum_{a\leq k <q} (-1)^{k-a}\binom{k}{a} \binom{x}{k}
=\sum_{0\leq k<q} \sum_{0\leq a\leq k} (-1)^{k-a} \binom{k}{a} \binom{x}{k}=1,\]
since the dimension of $Dist((G_m)_r)$ is $q=p^r$, and it equals the number of $X_a$.
\end{proof}

\subsection{$m\geq 1$}

For each $0\leq a<q$ denote by 
\[X_{i,a}=\sum_{k=a}^{q-1} (-1)^{k-a}\binom{k}{a} \binom{x_i}{k}\]
and 
\[h_a=h_{a_1, \ldots, a_m}= \prod_{i=1}^m X_{i,a_i}.\]

The multiplicative structure of $Dist(T_r)$ is described in the following lemma.

\begin{pr}\label{torusmn}
For each $r$, $Dist(T_r)$ is a separable algebra. In fact, every $h_{a_1, \ldots, a_m}$ is a primitive idempotent, and they together form a complete orthogonal decomposition of unity.
\end{pr}
\begin{proof}
Since $Dist(T_r)$ is a direct product of $m$ copies of $Dist((G_m)_r)$, 
the claim follows from Lemma \ref{torus1}.
\end{proof}

\section{Central elements in $Dist(GL(m|n))$}
We follow the notations from \cite{zubmar}.
Fix an odd element $e_{ij}\in gl(m|n)$ such that $1\leq i\leq m < j\leq m+n$. 
Let $\lambda=(\lambda_1, \ldots, \lambda_{m+n})$ be a weight of $G=GL(m|n)$. 
Denote the binomial elements $\binom{e}{\lambda}=\prod_{i=1}^{m+n} \binom{e_{ii}}{\lambda_i}$ and divided powers elements $e_{ij}^{(t)}=\frac{e^t_{ij}}{t!}$ for $t\geq 0$.
\begin{lm}\label{commutators} For $1\leq i\leq m < j\leq m+n$ and $t>0$ there are the following commutator formulas.
\begin{enumerate}
\item If $1\leq k\leq m< l\leq m+n$, then $[e_{kl}, e_{ij}]=0$;
\item If $1\leq k\neq l\leq m$ or $m+1\leq k\neq l \leq m+n$, then 
\[[e_{kl}^{(t)}, e_{ij}]=\delta_{il}e_{kl}^{(t-1)}e_{k j}-\delta_{k j}e^{(t-1)}_{kl}e_{il}=\delta_{il}e_{kj}e_{kl}^{(t-1)}-\delta_{k j}e_{il}e^{(t-1)}_{kl};\]
\item If $1\leq l\leq m < k\leq m+n$, then $[e_{kl}, e_{ij}]=\delta_{il}e_{kj}+\delta_{kj}e_{il}$.
\end{enumerate}
\end{lm}

Using PBW theorem, write the basis elements of $Dist(G)$ as products in the order such that odd factors $e_{uv}$, where $u> v$ appear first; followed by even factors
$e_{uv}^{t_{uv}}$ where $u>v$; basis elements $\binom{e}{\mu}$ of $Dist(T)$; even factors $e_{uv}^{d_{uv}}$ where $u<v$; and odd factors $e_{uv}$ where $u<v$ appearing last.
The basis elements $p_{t,\mu,d}$ of $Dist(G)$ can be written in the form 
\[\prod_{1\leq l\leq m< k\leq m+n}e_{kl}^{t_{kl}}
\prod_{\substack{1\leq l< k\leq m\\ m+1\leq l< k\leq m+n}}e_{kl}^{(t_{kl})}
\binom{e}{\mu}
\prod_{\substack{1\leq k< l\leq m\\m+1\leq k<l\leq m+n}}e_{kl}^{(d_{kl})}
\prod_{1\leq k\leq m< l\leq m+n}e_{kl}^{d_{kl}}.\]

Assume that a central element $c$ of $Dist(G)$ has the form
\begin{equation}\label{cc}
c=\sum_{\lambda} a_{\lambda}\binom{e}{\lambda}+\sum_{t, \mu, d, \sum t_{kl}+\sum d_{kl}>0}a_{t, \mu, d}p_{t,\mu,d}.\end{equation}
It is easy to see that the image $h(c)$ of $c$ under the Harish-Chandra morphism $h$ is
\[h(c)=\sum_{\lambda} a_{\lambda}\binom{e}{\lambda}.\]

The algebraic supergroup $GL(m|n)$ contains a group subfunctor $\Sigma$ such that for every $A\in\mathsf{SAlg}_K$ the group $\Sigma(A)$ consists of matrices
\[\left(\begin{array}{cc}
A_{\sigma} & 0 \\
0 & B_{\tau}\end{array}\right),\]
where $A_{\sigma}$ is a monomial matrix corresponding to a substitution $\sigma\in\Sigma_m$, and $B_{\tau}$ is a monomial matrix corresponding to a substitution $\tau\in\Sigma_n$, respectively. In particular, $\Sigma(A)\simeq \Sigma_m\times\Sigma_n$. Since the supergroup $GL(m|n)$ acts on its Lie superalgebra $gl(m|n)$ by conjugations, the group subfunctor $\Sigma$ acts on $gl(m|n)$ in such a way that $(\sigma, \tau)\in \Sigma_m\times\Sigma_n\simeq \Sigma(A)$ maps $e_{ij}$ to $e_{\sigma(i), \tau(j)}$.
Therefore, each $h(c)$ is symmetric in $e_{11}, \ldots, e_{mm}$ and $e_{m+1, m+1}, \ldots , e_{m+n, m+n}$.

Using Lemma 7.7 from \cite{zubmar}, we obtain
\[[c, e_{ij}]=\sum_{\lambda} a_{\lambda}(\binom{e}{\lambda}-\binom{e-\epsilon_i+\epsilon_j}{\lambda})e_{ij}+[c-h(c), e_{ij}]=0 .\]

\begin{pr}\label{star}
Let $c$ be a central element of $Dist(G)$ written in the form (\ref{cc}). Then for every $1\leq i\leq m < j\leq m+n$ there is
\begin{equation}\tag{$\star_{ij}$} \ \sum_{\lambda} a_{\lambda}(\binom{e}{\lambda}-\binom{e-\epsilon_i+\epsilon_j}{\lambda})=(e_{ii}+e_{jj})f,
\end{equation}
for some $f\in Dist(T)$. 
\end{pr}
\begin{proof}
By Lemma \ref{commutators}, $[c-h(c), e_{ij}]$ equals a sum of scalar multiples of elements $q_{t, \mu, d}$ obtained from elements
$p_{t, \mu, d}$
by replacing only one factor $e_{kl}^{t_{kl}}$, $e_{kl}^{(t_{kl})}, \binom{e}{\mu}, e_{kl}^{(d_{kl})}$ or $e_{kl}^{(d_{kl})}$ by
\\
1) $[e_{kl}, e_{ij}]$ or $e_{kl}^{(t_{kl}-1)}[e_{kl}, e_{ij}]$ if $l<k$; \\
2) $(\binom{e}{\mu}-\binom{e-\epsilon_i+\epsilon_j}{\lambda})e_{ij}$;\\
3) $e_{kl}^{(d_{kl}-1)}[e_{kl}, e_{ij}]$ or $[e_{kl}, e_{ij}]$ if $k<l$; respectively.

Set $i=m, j=m+1$. We claim that the only summands $p$ of $c-h(c)$, such that $[p, e_{m, m+1}]$ has a non-zero component $q$ that belong to $Dist(T)e_{m, m+1}$, are
\[p=e_{m+1, m}\binom{e}{\mu}e_{m, m+1}.\]

Consider a summand $q$ of the type (2). Then our claim is clear if $\sum t_{kl}>0$.
Assume $\sum t_{kl}=0$, and therefore $\sum d_{kl}> 0$. 

First assume that in the expression for $p$, 
there is at least one even factor $e_{kl}^{(d_{kl})}$, where $d_{kl}>0$.
Since $|e_{kl}|=0$, 
\[e_{m, m+1}e_{kl}^{(d_{kl})}=e_{kl}^{(d_{kl})}e_{m, m+1}-\delta_{lm}e_{kl}^{(d_{kl}-1)}e_{k, m+1}+
\delta_{k, m+1}e_{kl}^{(d_{kl}-1)}e_{ml},\]
where either $k< m$ provided $l=m$, or $l> m+1$ provided $k=m+1$.

Assume there is another even factor $e_{uv}^{(d_{uv})}$, where $d_{uv}>0$, following $e_{kl}^{(d_{kl})}$ in the expression for $p$.
If $l=m$, then
\[e_{k, m+1}e_{uv}^{(d_{uv})} = e_{uv}^{(d_{uv})}e_{k, m+1}-\delta_{vk}e_{uv}^{(d_{uv}-1)}e_{u, m+1}+\delta_{m+1, u}e^{(d_{uv}-1)}_{uv}e_{kv},\]
where $u< m$ provided $v=k$. 

Analogously, if $k=m+1$, then
\[e_{ml}e_{uv}^{(d_{uv})}=e_{uv}^{(d_{uv})}e_{ml}-\delta_{vm}e_{uv}^{(d_{uv}-1)}e_{ul}+
\delta_{ul}e_{uv}^{(d_{uv}-1)}e_{mv},\]
where $v> m+1$ provided $u=l$. 

Continuing this process, one quickly sees that after each step there is a unique summand ending in $e_{m,m+1}$ and all other summands are
ending with odd factors $e_{uv}$, where $1\leq u\leq m<v \leq m+n$ and either $u<m$ or $v>m+1$.

Therefore all nonzero summands corresponding to the rewritten 
expression for $e_{m,m+1}\prod_{\substack{1\leq k< l\leq m\\ m+1\leq k<l\leq m+n}} e_{kl}^{(d_{kl})}$ 
are either $\prod_{\substack{1\leq k< l\leq m\\ m+1\leq k<l\leq m+n}} e_{kl}^{(d_{kl})}e_{m,m+1}$ or 
end with one of the factors $e_{uv}$, where $1\leq u\leq m<v \leq m+n$ and either $u<m$ or $v>m+1$.

Since $[e_{uv}, e_{rs}]=0$ when $1\leq u,r\leq m<v,s\leq m+n$, when we multiply summands from 
$e_{m,m+1}\prod_{\substack{1\leq k< l\leq m\\ m+1\leq k<l\leq m+n}} e_{kl}^{(d_{kl})}$ 
by odd elements $\prod_{1\leq k\leq m< l\leq m+n}e_{kl}^{d_{kl}}$ we either get zero if one of the odd elements is repeated or otherwise, they are linear combinations of
elements of our basis of $Dist(G)$ that do not belong to $Dist(T)e_{m,m+1}$.

Now assume that in the expression for $p$, there is no even factor $e_{kl}^{(d_{kl})}$, where $d_{kl}>0$. Then $q$ is either zero, or otherwise contains 
at least two different odd factors of type $e_{u,v}$ where $u<v$ and therefore does not belong to $Dist(T)e_{m,m+1}$.

We leave it for the reader to verify that similar arguments work for summands $q$ of type (3).

Now, consider a summand $q$ of the type (1). The only non-zero summands correspond to $e_{k, m}$, where $k\geq m+1$ and $e_{m+1, l}$, where $l\leq m$. 
If $k> m+1$, then in the monomial $p_{t, \mu, d}$ the factor $e_{k,m}$ will be replaced by the even element $e_{k, m+1}$. Since $Dist(U^-)^+$ is a 
two-sided (super)ideal of $Dist(U^-)$, the result of such substitution belongs to $Dist(U^-)^+ x$, where
$x=\prod_{\substack{1\leq k< l\leq m\\m+1\leq k<l\leq m+n}}e_{kl}^{(d_{kl})}
\prod_{1\leq k\leq m< l\leq m+n}e_{kl}^{d_{kl}}$. Thus the claim follows. Analogous arguments work if $l< m$.

Therefore it remains to consider the case $k=m+1, l=m$. In this case, in the monomial $p_{t, \mu, d}$ the factor $e_{m+1, m}$ will be replaced by the element $e_{mm}+e_{m+1, m+1}$. 
If $\sum_{1\leq l< k\leq m+n}t_{kl}> 1$, then the claim is again evident. 

Otherwise, $\prod_{1\leq l\leq m< k\leq m+n}e_{kl}^{t_{kl}}
\prod_{\substack{1\leq l< k\leq m\\ m+1\leq l< k\leq m+n}}e_{kl}^{(t_{kl})}=e_{m+1, m}$. Then $q$ belongs to $Dist(T)e_{m,m+1}$ only if
$\prod_{\substack{1\leq k< l\leq m\\m+1\leq k<l\leq m+n}}e_{kl}^{(d_{kl})}
\prod_{1\leq k\leq m< l\leq m+n}e_{kl}^{d_{kl}}=e_{m, m+1}$.

Consequently,
\[\sum_{\lambda} a_{\lambda}(\binom{e}{\lambda}-\binom{e-\epsilon_m+\epsilon_{m+1}}{\lambda})e_{m, m+1}-\]\[\sum_{\mu}a_{t, \mu, d}(e_{mmm}+e_{m+1, m+1})
\binom{e}{\mu}e_{m, m+1}=0,\]
where $t_{kl}=1$ if $k=m+1, l=m$ and $t_{kl}=0$ otherwise, and symmetrically, $d_{kl}=1$ if $k=m, l=m+1$ and $d_{kl}=0$ otherwise.

This implies
\[\sum_{\lambda} a_{\lambda}(\binom{e}{\lambda}-\binom{e-\epsilon_m+\epsilon_{m+1}}{\lambda})=\]\[(e_{mm}+e_{m+1, m+1})\sum_{\mu}a_{t, \mu, d}
\binom{e}{\mu}=(e_{mm}+e_{m+1, m+1})x\]
for some $x\in Dist(T)$.

Since $h(c)$ is symmetric in both $e_{ii}$ and $e_{jj}$ for $1\leq i\leq m$ and $m+1\leq j\leq m+n$, the identity $(\star_{ij})$ holds for any $1\leq i\leq m< j \leq m+n$.
\end{proof}

\begin{lm}
The condition $(\star)_{ij}$ is equivalent to the system of equations
\[-\sum_{\substack{\pi_k=\lambda_k \mbox{ for } k\neq i, j\\ \lambda_i\leq\pi_i \\ \lambda_j\leq\pi_j\leq\lambda_j+1 \\ \pi_i\neq\lambda_i \ \mbox{or} \ \pi_j\neq\lambda_j}} (-1)^{\pi_i-\lambda_i}a_{\pi}=(\lambda_i+\lambda_j)b_{\lambda}+\lambda_i b_{\lambda-\epsilon_i}+\lambda_j b_{\lambda-\epsilon_j},\]
where $\lambda$ runs over all $(m+n)$-dimensional vectors with non-negative integer entries.
\end{lm}
\begin{proof}
Generalizing the formula in the proof of Proposition 8.3 in \cite{zubmar} we obtain
\[\binom{e-\epsilon_i+\epsilon_j}{\lambda}=\sum_{\substack{\beta_k=\lambda_k \mbox{ for } k\neq i, j\\ 0\leq\beta_i\leq\lambda_i\\ \lambda_j-1\leq\beta_j\leq\lambda_j}}(-1)^{\lambda_i-\beta_i}\binom{e}{\beta}.\]
Therefore the left-hand side of $(\star_{ij})$ equals
\[-\sum_{\lambda}a_{\lambda}(\sum_{\substack{\beta_k=\lambda_k \mbox{ for } k\neq i, j\\ 0\leq\beta_i\leq\lambda_i\\ \lambda_j-1\leq\beta_j\leq\lambda_j\\
\beta_i\neq\lambda_i \ \mbox{or} \ \beta_j\neq\lambda_j}}(-1)^{\lambda_i-\beta_i}\binom{e}{\beta})\]
which can be rewritten as
\[-\sum_{\lambda}(\sum_{\substack{\pi_k=\lambda_k \mbox{ for } k\neq i, j\\ \lambda_i\leq\pi_i\\ \lambda_j\leq\pi_j\leq\lambda_j+1\\
\pi_i\neq\lambda_i \ \mbox{or} \ \pi_j\neq\lambda_j}} (-1)^{\pi_i-\lambda_i}a_{\pi})\binom{e}{\lambda}.\]
Denote by $(e_{ii}+e_{jj})\sum_{\lambda}b_{\lambda}\binom{e}{\lambda}$ the right-hand side of $(\star)_{ij}$. 
It equals
\[\sum_{\lambda}b_{\lambda}((\lambda_i+1)\binom{e}{\lambda+\epsilon_i}+\lambda_i\binom{e}{\lambda}+(\lambda_j+1)\binom{e}{\lambda+\epsilon_j}+\lambda_j\binom{e}{\lambda}).\]

The last expression can be rewritten as
\[\sum_{\lambda}[\lambda_i(b_{\lambda}+b_{\lambda-\epsilon_i})+\lambda_j(b_{\lambda}+b_{\lambda-\epsilon_j})]\binom{e}{\lambda},\]
where the coefficient $b_{\lambda-\epsilon_i}$ or $b_{\lambda-\epsilon_j}$ is zero if $\lambda_i=0$ or $\lambda_j=0$, respectively. 
One can transform it as
\[\sum_{\lambda}[(\lambda_i+\lambda_j)b_{\lambda}+\lambda_i b_{\lambda-\epsilon_i}+\lambda_j b_{\lambda-\epsilon_j}]\binom{e}{\lambda}.\]

Comparison of the above expressions yields the stated system of equations.
\end{proof}
Symmetrically, one can record the system of equations, as above, of type $(\star)_{ji}$.


\section{Supersymmetric elements}

Let $f\in Dist(T)$. Denote by $s_{ij}(f)$ the element obtained from $f$ after replacing every $x_i$ by $x_i-1$ and every $y_j$ by $y_j+1$ and by $\Phi_{ij}$ the difference operator sending 
$f$ to $f-s_{ij}(f)$. Explicitly, 
\[\begin{aligned}&\Phi_{ij}(f(x_1, \ldots, x_m,y_1, \ldots, y_n))\\
&=f(x_1, \ldots, x_i-1, \ldots, x_m,y_1, \ldots, y_j+1, \ldots, y_n)-f(x_1, \ldots, x_m,y_1, \ldots, y_n)
\end{aligned}\]

\begin{df}
We say that $f\in Dist(T)$ satisfies the condition $(\dagger_{ij})$ if there is $f_{ij}'\in Dist(T)$ such that 
\[\tag{$\dagger_{ij}$} \Phi_{ij}(f)=(x_i+y_j)f_{ij}'.\]

If $f$ is symmetric with respect to the terms $x_1, \ldots, x_m$ and $y_1, \ldots, y_n$ separately, and it satisfies all conditions $(\dagger_{ij})$ for every $i$ and $j$, 
then we call $f$ {\it supersymmetric}. The subspace of all supersymmetric elements in $Dist(T)$ will be denoted by $SS$.

The set of all $f\in Dist(T_r)\cap SS$ for which there are $f_{ij}'\in Dist(T_r)$ such that  $\Phi_{ij}(f)=(x_i+y_j)f_{ij}'$ will be denoted by $SS_r$.
\end{df} 

Let us remark that if $f$ is symmetric with respect to the terms $x_1, \ldots, x_m$ and $y_1, \ldots, y_n$ separately, and it satisfies the conditions $(\dagger_{11})$, 
then it is supersymmetric since it automatically satisfies all conditions $(\dagger_{ij})$ for every $i$ and $j$.

\begin{rem} 
As mentioned in the introduction, we want to find a minimal set of generators of $SS_r$ supersymmetric elements over the Frobenius kernel $T_r$. This is an interesting problem on its own. However, we expect that each supersymmetric element is the image $h(c)$ for some $c\in Z$, that is, the equality $h(Z)=SS$ is valid for general $GL(m|n)$. If this is confirmed, then it would give crucial information about the structure of the center $Z$, analogous to the classical Harish-Chandra theory. 
\end{rem}

\begin{df}
For an algebra $A$, any $K$-linear map $d: A \to A$ satisfying  
\[d(f_1f_2)=(df_1)f_2+f_1(df_2)-(df_1)(df_2)\]
will be called a {\it quasiderivation} of $A$.
\end{df}

\begin{lm}\label{lm2}
Each operator $\Phi_{ij}$ is a quasiderivation of $Dist(T)$.
\end{lm}
\begin{proof}
Set $f_1=\binom{x_i}{a_1}\binom{y_j}{b_1}$ and $f_2=\binom{x_i}{a_2}\binom{y_j}{b_2}$ and compute 
\[\begin{aligned}&(\Phi_{ij}f_1)f_2+f_1(\Phi_{ij}f_2)-(\Phi_{ij}f_1)(\Phi_{ij}f_2)\\
&=(\binom{x_i}{a_1}\binom{y_j}{b_1}-\binom{x_i-1}{a_1}\binom{y_j+1}{b_1})\binom{x_i}{a_2}\binom{y_j}{b_2}\\
&+\binom{x_i}{a_1}\binom{y_i}{b_1}(\binom{x_i}{a_2}\binom{y_j}{b_2}-\binom{x_i-1}{a_2}\binom{y_j+1}{b_2})\\
&-(\binom{x_i}{a_1}\binom{y_j}{b_1}-\binom{x_i-1}{a_1}\binom{y_j+1}{b_1})(\binom{x_i}{a_2}\binom{y_j}{b_2}-\binom{x_i-1}{a_2}\binom{y_j+1}{b_2})
\end{aligned}\]
\[\begin{aligned}&=\binom{x_i}{a_1}\binom{x_i}{b_2}\binom{y_j}{b_1}\binom{y_j}{b_2}-\binom{x_i-1}{a_1}\binom{x_i}{a_2}\binom{y_j+1}{b_1}\binom{y_j}{b_2}\\
&+\binom{x_i}{a_1}\binom{x_i}{a_2}\binom{y_j}{b_1}\binom{y_j}{b_2}-\binom{x_i}{a_1}\binom{x_i-1}{a_2}\binom{y_j}{b_2}\binom{y_j+1}{b_2}\\
&-\binom{x_i}{a_1}\binom{x_i}{a_2}\binom{y_j}{b_1}\binom{y_j}{b_2}+\binom{x_i}{a_1}\binom{x_i-1}{a_2}\binom{y_j}{b_1}\binom{y_j+1}{b_2}
+\binom{x_i-1}{a_1}\binom{x_i}{a_2}\binom{y_j+1}{b_1}\binom{y_j}{b_2}\\
&-\binom{x-1}{a_1}\binom{x_i-1}{a_2}\binom{y_j+1}{b_1}\binom{y_j+1}{b_2}\\
&=\binom{x_i}{a_1}\binom{x_i}{a_2}\binom{y_j}{b_1}\binom{y_j}{b_2}-\binom{x-1}{a_1}\binom{x_i-1}{a_2}\binom{y_j+1}{b_1}\binom{y_j+1}{b_2}\\
&=\Phi(f_1f_2)
\end{aligned}\]
\end{proof}
\vspace{-1em}

\begin{rem}
There is a connection between supersymmetric elements in the distribution algebra of the torus $T$ defined above and the classical supersymmetric polynomials introduced in 
\cite{stem} and further investigated in 
 \cite{kanttrish}, \cite{lascala} and \cite{grmazu}. The above difference operator $\Phi_{ij}$ 
is related to the derivation 
\[(\frac{d}{dx_i}+\frac{d}{dy_j})f(x_1, \ldots, x_m,y_1, \ldots, y_n).\]

If we replace elements of $Dist(T)$ by polynomials in $x_1, \ldots, x_m,y_1, \ldots, y_n$, the quasiderivation $\Phi_{ij}$ by the derivation $\frac{d}{dx_i}+\frac{d}{dy_j}$ and make a substitution $y_j\mapsto -y_j$, the property $(\dagger_{ij})$ is transformed to 
\[(\frac{d}{dx_i}+\frac{d}{dy_j})(f(x_1, \ldots, x_m,y_1, \ldots, y_n))=(x_i-y_j)f'_{ij}(x_1, \ldots, y_m,y_1, \ldots, y_n),\]
which together with $f(x_1, \ldots, x_m,y_1, \ldots, y_n)$ been invariant under permutations of variables $x_1, \ldots, x_m$ and $y_1, \ldots, y_n$ separately
gives the definition of when the polynomial $f(x_1, \ldots, x_m,y_1, \ldots, y_n)$ is supersymmetric.
\end{rem}

Our goal is to describe all supersymmetric elements in $Dist(T)$. 
\begin{lm}\label{lm3}
$SS=\lim\limits_{\longrightarrow} SS_r$.
Additionally, $SS$ is a subalgebra of $Dist(T)$ and $SS_r$ is a subalgebra of $Dist(T_r)\cap SS$.
\end{lm}
\begin{proof}
If $f$ is supersymmetric, then there are elements $f_{ij}'$ such that $\Phi_{ij}(f)=(x_i+y_j)f_{ij}'$. We can choose $q$ large enough so that all exponents $a_i,b_j$ in the monomials 
$\binom{x_1}{a_1}\ldots \binom{x_m}{a_m}\binom{y_1}{b_1}\ldots \binom{y_n}{b_n}$ appearing in $f$ and all $f_{ij}'$ satisfy $0\leq a_i,b_j<q$. Then $f\in SS_r$.

If $f_1,f_2\in SS$, then $\Phi_{ij}(f_1)$ and $\Phi_{ij}(f_2)$ are multiples of $x_i+y_j$ and by Lemma \ref{lm2} so is $\Phi_{ij}(f_1f_2)$, showing that $f_1f_2\in SS$.
Therefore $SS$ is a subalgebra of $Dist(T)$. 

Analogously, $SS_r$ is a subalgebra of $Dist(T_r)$, hence of $Dist(T_r)\cap SS$.
\end{proof}

As a consequence of the above lemma, from now on we will work inside a fixed $Dist(T_r)$.


Next, we will show that the condition $(\dagger_{st})$ can be recognized by working over the torus of $GL(1|1)$ corresponding to variables $x_s$ and $y_t$.

Let $f\in Dist(T)$ and write
\[f=\sum_{(i_1, \ldots, i_m,j_1, \ldots, j_{n})} a_{i_1, \ldots, i_{m},j_1, \ldots, j_{n}} \binom{x_1}{i_1} \ldots \binom{x_m}{i_m}\binom{y_1}{j_1} \ldots \binom{y_n}{j_n}.\]
Fix $1\leq s \leq m$ and $1\leq t\leq n$.

Write 
\[\begin{aligned}f= &\sum_{(i_1, \ldots, \widehat{i_{s}}, \ldots, i_m, j_1, \ldots, \widehat{j_{t}}, \ldots, j_n)}  
\binom{x_1}{i_1}\ldots \widehat{\binom{x_s}{i_s}} \ldots \binom{x_m}{i_m}  \binom{y_1}{j_1} \ldots \widehat{\binom{y_s}{j_s}} \ldots \binom{y_n}{j_n}\times\\
&f_{(i_1, \ldots, \widehat{i_{s}}, \ldots, i_m, j_1, \ldots, \widehat{j_{t}}, \ldots, j_n)},
\end{aligned}\]
where 
\[f_{(i_1, \ldots, \widehat{i_{s}}, \ldots, i_m, j_1, \ldots, \widehat{j_{t}}, \ldots, j_n)}(x_s,y_t)=
\sum_{(i_s,j_t)} a_{i_1, \ldots, i_m,j_1, \ldots, j_n} \times \binom{x_s}{i_s} \binom{y_t}{j_t}\]
is obtained by ``freezing'' variables 
$x_1, \ldots, \widehat{x_s}, \ldots, x_m,y_1, \ldots, \widehat{y_t}, \ldots y_n.$

\begin{lm}\label{induction}
An element $f$ in $Dist(T)$ satisfies the condition $(\dagger_{st})$ if and only if all summands $f_{(i_1, \ldots, \widehat{i_{s}}, \ldots, i_m, j_1, \ldots, \widehat{j_{t}}, \ldots, j_n)}$, 
considered as elements in variables $x_s,y_t$ only, satisfy the condition $(\dagger_{st})$.
\end{lm}
\begin{proof}
Let $f'\in Dist(T)$ and 
\[\begin{aligned}f'= &\sum_{(i_1, \ldots, \widehat{i_{s}}, \ldots, i_m, j_1, \ldots, \widehat{j_{t}}, \ldots, j_n)}  
\binom{x_1}{i_1}\ldots \widehat{\binom{x_s}{i_s}} \ldots \binom{x_m}{i_m}  \binom{y_1}{j_1} \ldots \widehat{\binom{y_s}{j_s}} \ldots \binom{y_n}{j_n}\times\\
&f'_{(i_1, \ldots, \widehat{i_{s}}, \ldots, i_m, j_1, \ldots, \widehat{j_{t}}, \ldots, j_n)}.
\end{aligned}\]

Since \[\begin{aligned}\Phi_{st}(f)= &\sum_{(i_1, \ldots, \widehat{i_{s}}, \ldots, i_m, j_1, \ldots, \widehat{j_{t}}, \ldots, j_n)}  
\binom{x_1}{i_1}\ldots \widehat{\binom{x_s}{i_s}} \ldots \binom{x_m}{i_m}  \binom{y_1}{j_1} \ldots \widehat{\binom{y_s}{j_s}} \ldots \binom{y_n}{j_n}\times\\
&\Phi_{st}(f_{(i_1, \ldots, \widehat{i_{s}}, \ldots, i_m, j_1, \ldots, \widehat{j_{t}}, \ldots, j_n)}),
\end{aligned}\]
and
\[\begin{aligned}(x_s+y_t)f'= &\sum_{(i_1, \ldots, \widehat{i_{s}}, \ldots, i_m, j_1, \ldots, \widehat{j_{t}}, \ldots, j_n)}  
\binom{x_1}{i_1}\ldots \widehat{\binom{x_s}{i_s}} \ldots \binom{x_m}{i_m}  \binom{y_1}{j_1} \ldots \widehat{\binom{y_s}{j_s}} \ldots \binom{y_n}{j_n}\times\\
&(x_s+y_t)f'_{(i_1, \ldots, \widehat{i_{s}}, \ldots, i_m, j_1, \ldots, \widehat{j_{t}}, \ldots, j_n)},
\end{aligned}\]
the claim follows.
\end{proof}

\section{Special supersymmetric elements}

\begin{lm}\label{lm4}
We have 
\[(x_i+y_j)h_{a_1, \ldots, a_m|b_1, \ldots,b_n}=(a_i+b_j)h_{a_1, \ldots, a_m|b_1, \ldots,b_n}\]
and
\[(a_i+b_j)\Phi_{ij}(h_{a_1, \ldots, a_m|b_1, \ldots,b_n})=(x_i+y_j)\Phi_{ij}(h_{a_1, \ldots, a_m|b_1, \ldots,b_n}).\]
\end{lm}
\begin{proof}
If $a_i>0$, then 
\[\begin{aligned}&x_iX_{i,a_i}=a_i\binom{x_i}{a_i}+\sum_{k=a_i+1}^{q-1} [(-1)^{k-a_i}\binom{k}{a_i}+(-1)^{k+1-a_i}\binom{k+1}{a_i}](k+1)\binom{x_i}{k+1}\\
&=a_i\binom{x_i}{a_i}+\sum_{k=a_i+1}^{q-1} (-1)^{k+1-a_i}\binom{k}{a_i-1}(k+1)\binom{x_i}{k+1}\\
&=a_i\binom{x_i}{a_i}+\sum_{k=a_i+1}^{q-1} (-1)^{k+1-a_i}a_i\binom{k+1}{a_i}\binom{x_i}{k+1}\\
&=a_iX_{i,a_i}.
\end{aligned}\]
If $a_i=0$, then 
\[x_iX_{i,0}=\sum_{k=1}^{q-1} [(-1)^{k}+(-1)^{k+1}](k+1)\binom{x_i}{k+1}=0=a_iX_{i,0}.\]

Analogously, $y_jY_{j,b_j}=b_jY_{j,b_j}$.

Therefore $(x_i+y_j)X_{i,a_i}Y_{j,b_j}=(x_iX_{i,a_i})Y_{j,b_j}+X_{i,a_i}(y_jY_{j,b_j})=(a_i+b_j)X_{i,a_i}Y_{j,b_j}$, 
which implies 
\[(x_i+y_j)h_{a_1, \ldots, a_m|b_1, \ldots,b_n}=(a_i+b_j)h_{a_1, \ldots, a_m|b_1, \ldots,b_n}.\]
Applying the map $s_{ij}$ we obtain 
\[(x_i-1+y_j+1)s_{ij}(h_{a_1, \ldots, a_m|b_1, \ldots,b_n})=(a_i+b_j)s_{ij}(h_{a_1, \ldots, a_m|b_1, \ldots,b_n})\]
which implies 
\[(x_i+y_j)\Phi_{ij}(h_{a_1, \ldots, a_m|b_1, \ldots,b_n})=(a_i+b_j)\Phi_{ij}(h_{a_1, \ldots, a_m|b_1, \ldots,b_n}).\]
\end{proof}

As a corollary of the last lemma, we get the following statement.
\begin{cor}\label{cor1}
The set of all elements in $Dist(T_r)$ that are multiples of $x_i+y_j$ has a basis consisting of all elements $h_{a_1, \ldots, a_m|b_1, \ldots,b_n}$ such that $p$ does not divide $a_i+b_j$.
\end{cor}

Denote by $\tilde{h}_{a_1, \ldots, a_m|b_1, \ldots,b_n}$ the symmetrizer of $h_{a_1, \ldots, a_m|b_1, \ldots,b_n}$ with respect to the action of the symmetric group
$\Sigma_m$ on $a_1, \ldots, a_m$ and of the symmetric group $\Sigma_n$ on $b_1, \ldots, b_n$. 

\begin{pr}\label{pr1}
If no $a_i+b_j$, for $1\leq i \leq m$ and $1\leq j\leq n$, is divisible by $p$, then $\tilde{h}_{a_1, \ldots, a_m|b_1, \ldots,b_n}\in SS$.
\end{pr}
\begin{proof}
By definition, $\tilde{h}_{a_1, \ldots, a_m|b_1, \ldots,b_n}$ is symmetric with respect to $x_1, \ldots, x_m$ and $y_1, \ldots, y_n$ separately. 
It satisfies the conditions $(\dagger_{ij})$ by Lemma \ref{lm4}.
\end{proof}

An element $h_{a_1, \ldots, a_m|b_1, \ldots,b_n}$, such that all $a_i+b_j$ are not divisible by $p$, will be called {\it special}. An element 
$h_{a_1, \ldots, a_m|b_1, \ldots,b_n}$, such that $p$ divides some $a_i+b_j$, will be called {\it ordinary}.

For $0\leq a<q$, denote by $H_a$ the sum of all ordinary elements $h_{a_1, \ldots, a_m|b_1, \ldots,b_n}\in Dist(T_r)$ such that 
\[a_1+\ldots + a_m+b_1+\ldots +b_n  \equiv a \pmod q.\]

\begin{pr}\label{pr3}
The elements $H_a$ are supersymmetric.
\end{pr}
\begin{proof}
Clearly, if $H_a$ contains a summand $h_{a_1, \ldots, a_m|b_1, \ldots,b_n}$, then it contains all summands obtained by permuting indices $a_1, \ldots, a_m$ and $b_1, \ldots, b_n$ separately. Therefore it remains to prove that $H_a$ satisfies the condition $(\dagger_{11})$. 
Split all summands in $H_a$ into groups consisting of those $h_{a_1, \ldots, a_m|b_1, \ldots,b_n}$ for which $a_2, \ldots, a_m; b_2, \ldots, b_n$ are fixed.
Then $a_1+b_1$ modulo $q$ remains constant within a group, say $a_1+b_1\equiv k \pmod q$. Then the sum of all term in a group equals 
\[\sum_{a_1+b_1\equiv k\pmod q} h_{a_1|b_1} h_{a_2,\ldots, a_m|b_2, \ldots, b_n},\] 
which by Proposition \ref{pr2} satisfies the condition $(\dagger_{11})$. Therefore, by Lemma \ref{induction}, 
$H_a$ also satisfies the condition $(\dagger_{11})$.
\end{proof}

\section{Supersymmetric elements for $GL(m|1)$}

In light of the Lemma \ref{induction}, it is essential to describe all supersymmetric elements for $GL(1|1)$. 

\subsection{$GL(1|1)$} Assume $G=GL(1|1)$.

Some elements in $SS_r$ were described in \cite{zubmar}, as elements of $h(Z_r)$ of the image of the center $Z_r$ of $Dist(G_r)$ under the map $h$.
Denote by $\overline{a}$ the remainder after the division of $a$ by $q$.

\begin{pr}\label{pr2}
The generators of $h(Z_r)$ for $GL(1|1)$ are given as follows:

-$h_{a|b}$ for $a+b$ not divisible by $p$;

-$\sum_{i=0}^q h_{i|\overline{pl-i}}$ where $l=0, \ldots, \frac{q}{p}-1$.
\end{pr}
\begin{proof}
The proof follows from Theorem 9.10 of \cite{zubmar}.
\end{proof}

The elements listed in the above Proposition are generators of $SS_r$ for $GL(1|1)$ due to the following Proposition.

\begin{pr}\label{rovnost}
Let $G=GL(1|1)$. Then $SS_r=h(Z_r)$ for $r>0$.
\end{pr}
\begin{proof}
Since $h(Z_r)\subseteq SS_r$ by Proposition \ref{star}, and Proposition \ref{pr2} together with Proposition \ref{torusmn} imply that $dim(h(Z_r))=q(q-\frac{q}{p})+\frac{q}{p}$, 
it is enough to verify that the dimension of $SS_r$ also equals $q(q-\frac{q}{p})+\frac{q}{p}$.

Denote by $\mathcal{B}$ the basis of $Dist(T_r)$ consisting of elements $\mathfrak{b}_{ij}=\binom{x}{i}\binom{y}{j}$ for $0\leq i,j<q$ listed in the lexicographic order on $(ij)$, and 
write $f\in Dist(T_r)$ in the form $f=\sum_{i,j=0}^{q-1} a_{ij} \mathfrak{b}_{ij}$.

The matrix $C$ of the linear map $L$ on $Dist(T_r)$ corresponding to the multiplication by $x+y$, given by
$\binom{x}{i}\binom{y}{j}\mapsto (x+y)\binom{x}{i}\binom{y}{j}$, in the basis $\mathcal{B}$ 
has coefficients given as $c_{i,j;i,j}=i+j$, $c_{i+1,j;ij}=i+1$, $c_{i,j+1;ij}=j+1$ and all other coefficients vanish.

By Corollary \ref{cor1}, the rank of the matrix $C$ is $q(q-\frac{q}{p})$, hence its nullity is $\frac{q^2}{p}$.
Then $L(Dist(T_r))=Col(C)$ is the orthogonal complement of the space $R=Nul(C^T)$.

Using formulas $\binom{x-1}{i}=\sum_{l=0}^{i} (-1)^{i-l} \binom{x}{l}$ and $\binom{y+1}{j}=\binom{y}{j-1}+\binom{y}{j}$ for $j>0$, we derive that 
$f$ belongs to $\Phi_{1,1}(Dist(T_r))$ if and only if $\sum_{i+j=q-2+k} a_{ij} =0$. Let $D$ be a matrix of $\Phi_{1,1}$ in the basis $\mathcal{B}$.
Then $\Phi_{1,1}(Dist(T_r))=Col(D)$, the rank of $D$ is $q^2-q$ and its nullity is $q$. The basis $\mathcal{R}'$ of the space $R'=Nul(D^T)$ consists of vectors 
$\rho'_k$ for $k=1, \ldots, q$ such that its entries at the positions $(ij)$ are $1$ if $i+j=q-2+k$, and vanish otherwise.
 
Denote by $r_{ij}$ the row of matrix $C$ consisting of entries $c_{i,j;u,v}$ for $0 \leq u,v<q$, and by $r_k=\sum_{i+j=q-2+k} r_{ij}$ for $k=1, \ldots, q$.
Then the entries of $r_k$ at positions $(ij)$ equal to $q-1+k$ if $i+j=q-3+k$; equal to $q-2+k$ if $i+j=q-2+k$; and vanish otherwise. 

If a vector $\sum_{k=1}^q A_k\rho'_k\in R'$ also belongs to $R$, then $\sum_{k=1}^q A_k r_k=0$.
Using the above description of $r_k$, this is equivalent to the linear system 
\[A_k(q-2+k)+A_{k+1}(q+k)=0\] for $k=1, \ldots, q-1$ and $A_q=0$. From this it is clear that $dim(R\cap R')=\frac{q}{p}$.

Since $dim(R)= q$, $dim(R')=\frac{q^2}{p}$ and $dim(R\cap R')=\frac{q}{p}$, we have $dim(R+R')=q+\frac{q^2}{q}-\frac{q}{p}$.
Since $\Phi_{1,1}(Dist(T_r))\cap L(Dist(T_r))=Col(D)\cap Col(C)=R^{\perp}\cap R'^{\perp}=(R+R')^{\perp}$, we get that 
the dimension of the space $\Phi_{1,1}(Dist(T_r))\cap L(Dist(T_r))$ equals $q^2-q-\frac{q^2}{p}+\frac{q}{p}$. By adding the dimension $q$ of the kernel of $\Phi_{1,1}$, we obtain that the dimension of $SS_r$ is $q^2-\frac{q^2}{p}+\frac{q}{p}$, as claimed.
\end{proof}

\begin{rem}
It is worth mentioning that the elements $h_{a|b}$ and their sums appearing in Proposition \ref{pr2} underline that it is natural to consider elements of $Dist(T_r)$ as 
a linear combinations of elements $h_{a_1, \ldots, a_m|b_1, \ldots, b_n}$ from Proposition \ref{torusmn} and congruences of type 
$a_1+\ldots + a_m+ b_1+\ldots + b_n \equiv c \pmod q$. 
\end{rem}

To help the reader to build a geometric intuition and prepare for the counting arguments appearing later, we consider the cases $GL(2|1)$ and $GL(3|1)$ before the case $GL(m|1)$.

\subsection{$GL(2|1)$}

Assume $G=GL(2|1)$.

\begin{pr}
The basis of supersymmetric elements in $SS_1$ consists of special elements $\tilde{h}_{a_1,a_2|b}$ and the sums $H_a$. The dimension of $SS_1$ is $p\frac{p^2-p+2}{2}$.
\end{pr}
\begin{proof}
We visualize the form of summands of $H_a$ as follows. Since the sum $a_1+a_2+b\equiv a \pmod p$ for each summand $h_{a_1,a_2|b}$ of $H_a$ remains constant modulo $p$, we represent 
ordinary $h_{a_1,a_2|b}$ by a lattice point $(a_1,a_2)$ in the square $[0,p)\times [0,p)$. Then $H_a$ is a sum of $h_{a_1,a_2|b}$, where $a_1+a_2+b\equiv a \pmod p$ and 
the corresponding marked lattice points $(a_1,a_2)$ lie on two line segments given by $x_1=a$ (where $a_2+b\equiv 0 \pmod p$) and $x_2=a$ (where $a_1+b\equiv 0 \pmod p$)
intersecting at the point $(a,a)$. 
Therefore there are $2p-1$ summands in $H_a$, using up $2p-1$ ordinary elements. 
If $f\in SS_r$ has a nonzero coefficient at an ordinary element $h_{a_1,a_2|b}$, then all the coefficients along the lines $x_1=a$ and $x_2=a$ have to be the same. If we add up all of the corresponding elements, we obtain a scalar multiple of $H_a$ that can be split off from $f$. 

Since the number of special elements $\tilde{h}_{a_1,a_2|b}$ is $p\binom{p}{2}$, the dimension formula follows.
\end{proof}

\begin{pr}
The basis of supersymmetric elements in $SS_r$ consists of special elements $\tilde{h}_{a_1,a_2|b}$ and the sums $H_a$. The dimension of $SS_r$ is 
$\frac{p^{3r-2}(p-1)^2+p^{2r-1}(p-1)}{2}+p^r$.
\end{pr}
\begin{proof}
We need to explain how to scale up the argument from the previous Proposition from $p$ to $q$. We compute all expressions modulo $q$.

We visualize the form of summands of $H_a$ as follows. Since the sum $a_1+a_2+b\equiv a \pmod q$, for each summand $h_{a_1,a_2|b}$ of $H_a$, remains constant modulo $q$, we represent 
$h_{a_1,a_2|b}$ by a lattice point $(a_1,a_2)$ in the square $[0,q)\times [0,q)$. Then $H_a$ is a sum of $h_{a_1,a_2|b}$, where $a_1+a_2+b\equiv a \pmod q$ and 
the corresponding marked lattice points $(a_1,a_2)$ lie in the mesh given by lines $x_1=a$, $x_1=a+p$, \ldots, $x_1=a+q-p$ and $x_2=a$, $x_2=a+p$, \ldots, $x_2=a+q-p$. Break this mesh  
into $\frac{q}{p}\times \frac{q}{p}$ pieces bounded by squares of size $p$ obtained by shifting the lattice points inside the original square  $[0,p)\times [0,p)$.
In each such mesh there are $\frac{q^2}{p^2}(2p-1)$ marked lattice points corresponding to ordinary elements from $H_a$.
 
If $f\in SS_r$ has a nonzero coefficient at an ordinary element $h_{a_1,a_2|b}$, then all the coefficients corresponding to marked lattice points in the mesh as above have to be the same. If we add up all of the corresponding elements, we obtain a scalar multiple of $H_a$, that can be split off from $f$. 

Since the number of special elements $\tilde{h}_{a_1,a_2|b}$ is $q\binom{q-\frac{q}{p}+1}{2}$, the dimension formula follows.
\end{proof}

We will use an argument analogous to the above to scale up from $p$ to $q$ in general.


\subsection{$GL(3|1)$}

Assume $G=GL(3|1)$.

Define 
$T_{a,a'}$ for $0\leq a\leq a'<q$ to be a sum of $h_{a_1,a_2,a_3|b}$, where $a_1+a_2+a_3+b\equiv a+a' \pmod q$. 

We describe the generators of supersymmetric elements first in the case $q=p$, and then we scale it up to the case of general case analogously as above.

\begin{pr}
The basis of supersymmetric elements in $SS_1$ consists of special elements $\tilde{h}_{a_1,a_2,a_3|b}$ and elements $T_{a,a'}$ for $a\leq a'$.  The dimension of $SS_1$ is 
$p\binom{p+1}{3}+p+\binom{p}{2}$.
\end{pr}
\begin{proof}
We visualize the form of summands of $H_a$ as follows. Since the sum $a_1+a_2+a_3+b\equiv a \pmod p$ for each summand $h_{a_1,a_2,a_3|b}$ of $H_a$ remains constant modulo $p$, we represent ordinary $h_{a_1,a_2,a_3|b}$ by a lattice point $(a_1,a_2,a_3)$ in the square $[0,p)\times [0,p)\times[0,p)$. 

The marked lattice points $(a_1,a_2,a_3)$ corresponding to $T_{a,a}$ lie on three line segments given by 

$l_1: x_1=x_2=a$;

$l_2: x_1=x_3=a$;

$l_3: x_2=x_3=a$,

all intersecting at the point $(a,a,a)$. 
Therefore there are $3p-2$ summands in $T_{a,a}$, using up $3p-2$ ordinary elements. 

The marked lattice points $(a_1,a_2,a_3)$ corresponding to $T_{a,a'}$ for $a<a'$ lie on six line segments given by 

$l_1: x_2=a,  x_3=a'$; 

$l_2: x_1=a', x_2=a$;

$l_3: x_1=a', x_3=a$;

$l_4: x_2=a', x_3=a$;

$l_5: x_1=a, x_2=a'$;

$l_6: x_1=a,  x_3=a'$.

These line segments are organized in the shape of a hexagon $(l_1,l_2,l_3,l_4,l_5,l_6)$ such that neighboring line segments intersect at at unique point. 
There are $6(p-2)+6 = 6p-6$ summands in $T_{a,a'}$ for $a<a'$, using up $6p-6$ ordinary elements.

By an easy dimension count, all ordinary elements appear in precisely one of $T_{a,a'}$.

Assume that $f\in SS_1$ has a nonzero coefficient at an ordinary element $h_{a_1,a_2,a_3|b}$. 
Then this ordinary element belongs to one of $T_{a,a'}$, and the coefficients at $h_{c_1,c_2,c_3|d_1}$ along $T_{a,a'}$ have to be the same. If we add up all of the corresponding elements, we obtain a scalar multiple of $T_{a,a'}$ that can be split off from $f$.  

Since there are $p\binom{p+1}{3}$ special elements $\tilde{h}_{a_1,a_2, a_3|b}$, $p$ elements $T_{a,a}$ and $\binom{p}{2}$ elements $H_{a,a'}$ for $a<a'$, the dimension formula follows.
\end{proof}

\begin{pr}
The basis of supersymmetric elements in $SS_r$ consists of special elements $\tilde{h}_{a_1,a_2,a_3|b}$ and elements $T_{a,a'}$ for $a\leq a'$. The dimension of $SS_r$ is 
$q\binom{q-\frac{q}{p}+2}{3}+q+\frac{q}{p}\binom{p}{2}$.
\end{pr}
\begin{proof}
Scale up from $p$ to $q$ as before. Since there are $q\binom{q-\frac{q}{p}+2}{3}$ special elements $\tilde{h}_{a_1,a_2, a_3|b}$, $q$ elements $T_{a,a}$ and $\frac{q}{p}\binom{p}{2}$ 
elements $T_{a,a'}$ for $a<a'$, the dimension formula follows. 
\end{proof}

\subsection{$GL(m|1)$}

Assume $G=GL(m|1)$.

As before, in the beginning, we assume that $r=1$. 

We will show that a basis of supersymmetric elements composed of ordinary elements are labeled by canonical $(m+1)$-tuples $(a_1^{m_1}a_2^{m_2}\ldots a_s^{m_s}|0)$ defined by 
$a_1=0<a_2<\ldots <a_s<p$, $m_1, \ldots, m_s> 0$ and $m_1+\ldots+m_s=m$.

Starting from the canonical term $(a_1^{m_1}a_2^{m_2}\ldots a_s^{m_s}|0)$ we construct additional $s-1$ terms 
\[(a_ia_1^{m_1-1}a_2^{m_2}\ldots a_s^{m_s}|p-a_i)\] for $i=2, \ldots, s$ that will be called {\it intersection points}. If $m_1>1$, then the canonical term 
$(a_1^{m_1}a_2^{m_2}\ldots a_s^{m_s}|0)$ is also  called an intersection point.

All of the above intersection points are lattice points on the line segment given by equations

\[\begin{aligned}
x_1+y_1=&0,\\
x_2=&a_1,\ldots, x_{m_1}=a_1,\\
x_{m_1+1}=&a_2, \ldots, x_{m_1+m_2}=a_2, \\
&\ldots \\
x_{m-m_s+1}=&a_s, \ldots, x_m=a_s,
\end{aligned}\]
where each $0\leq x_i<p$ for $i=1, \ldots, s$ and $0\leq y_1<p$.

Together with the above intersection points we also consider all of their permutations under the action of the symmetric group $\Sigma_m$ on $m$ elements. This way we obtain 
$\binom{m}{m_1, \ldots, m_s}$ intersection points from $(a_1^{m_1}a_2^{m_2}\ldots a_s^{m_s}|0)$ if $m_1>1$, 
$\binom{m}{m_1-1,m_2+1, \ldots, m_s}$ intersection points from $(a_2a_1^{m_1-1}\ldots a_s^{m_s}|p-a_2)$ and so on until 
$\binom{m}{m_1-1, \ldots, m_s+1}$ intersection points from $(a_2a_1^{m_1-1}\ldots a_s^{m_s}|p-a_s)$.
The total number of intersection points is 
\[\begin{aligned}
&\delta_{m_1>1}\binom{m}{m_1, \ldots, m_s}+\binom{m}{m_1-1,m_2+1, \ldots, m_s}+ \ldots + \binom{m}{m_1-1, \ldots, m_s+1}\\
&=\binom{m}{m_1, \ldots, m_s}m_1(\delta_{m_1>1}\frac{1}{m_1}+\frac{1}{m_2+1}+\ldots +\frac{1}{m_s+1}).
\end{aligned}\]

If $m_1>1$, then each intersection point $(b_1\ldots b_m|0)$ permuted from $(a_1^{m_1}\ldots a_s^{m_s}|0)$ lies on the intersection of $m_1$ line segments given by 
equations $x_i+y_1=0$, where $i$ is such that $b_i=a_1$, and $x_l=b_l$ for $l\neq i$.

Each intersection point $(b_1\ldots b_m|p-a_j)$ permuted from $(a_ja_1^{m_1-1}\ldots a_s^{m_s}|p-a_j)$, where $j=2, \ldots s$, lies on the intersection of 
$n_j+1$ line segments given by equations $x_i+y_1=0$, where $i$ is such that $b_i=a_j$, and $x_l=b_l$ for $l\neq i$.

Since each line segment contains precisely $s-\delta_{m_1,1}$ intersection points, the total number of line segments is
\[\begin{aligned}&\frac{\delta_{m_1>1}m_1\binom{m}{m_1, \ldots, m_s}+(m_2+1)\binom{m}{m_1-1,m_2+1, \ldots, m_s}+ \ldots +(m_s+1)\binom{m}{m_1-1,m_2, \ldots, m_s+1}}{s-\delta_{m_1,1}}\\
&=m_1\binom{m}{m_1, \ldots, m_s}.\end{aligned}\]

The supersymmetric element $H_{a_1^{m_1}\ldots a_s^{m_s}|0}$ associated with $(a_1^{m_1}\ldots a_s^{m_s}|0)$ is obtained by adding up all terms $h_{b_1, \ldots, b_m|c}$, each with a simple multiplicity,
where the lattice point $(b_1, \ldots, b_m|c)$ appears on one of the above line segments. 
In the sum 
\begin{equation} \tag{*} \sum_{\text{line segments}} \quad \sum_{\text{lattice points}} h_{b_1,\ldots, b_s|c},\end{equation} 
taken over all line segments associated with $(a_1^{m_1}\ldots a_s^{m_s}|0)$ 
and lattice points appearing in them, all lattice points that are not intersection points appear with the coefficient one. 
If $m_1>1$, then the intersection points permuted from 
$(a_1^{m_1}\ldots a_s^{m_s}|0)$ appear with the coefficient $m_1$. The intersection points permuted from $(a_ja_1^{m_1-1}\ldots a_s^{m_s}|p-a_j)$, where $j=2, \ldots s$,
appear with the coefficient $m_j+1$. Therefore 
\[\begin{aligned}&H_{a_1^{m_1}\ldots a_s^{m_s}|0}=\sum_{\text{line segments}} \quad \sum_{\text{lattice points}} h_{b_1,\ldots, b_s|c} \\
&- (m_1-1)\sum_{\substack{\text{intersection points}\\ \text{permuted from}\\(a_1^{m_1}\ldots a_s^{m_s}|0)}} h_{b_1,\ldots, b_s|c}
-\sum_{j=2}^s m_j\sum_{\substack{\text{intersection points}\\ \text{permuted from}\\(a_ja_1^{m_1-1}\ldots a_s^{m_s}|p-a_j)}} h_{b_1,\ldots, b_s|c}.
\end{aligned}\]

\begin{rem}Note that each lattice point on the above line segments is ordinary. 
The number $m_1\binom{m}{m_1,\ldots, m_s}p$ of all ordinary lattice points (counted with multiplicities) appearing in the sum $(*)$
equals the product of the number of line segments and $p$. 
Since the number of ordinary intersection points (counted with multiplicities) associated with $(a_1^{m_1}\ldots a_s^{m_s}|0)$ is 
\[\begin{aligned}&\delta_{m_1>1}m_1\binom{m}{m_1, \ldots, m_s}+(m_2+1)\binom{m}{m_1-1,m_2+1, \ldots, m_s}+ \ldots \\
&+(m_s+1)\binom{m}{m_1-1,m_2, \ldots, m_s+1}= m_1\binom{m}{m_1,\ldots, m_s}(s-\delta_{m_1,1}).\end{aligned}\]
Therefore the total number of ordinary lattice points appearing in the sum $(*)$ that are not intersection points equals $m_1\binom{m}{m_1,\ldots, m_s}(p-s+\delta_{m_1,1})$.
\end{rem} 

\begin{pr}
\[\begin{aligned}&\sum_{\substack{canonical \\(a_1^{m_1}\ldots a_s^{m_s}|0)}}
\binom{m}{m_1,\ldots, m_s}m_1(p-s+\delta_{m_1,1}+\delta_{m_1>1}\frac{1}{m_1}+\frac{1}{m_2+1}+\ldots +\frac{1}{m_s+1})]=\\
&\sum_{i=0}^{m-1} (-1)^{i+m-1} \binom{m}{i}p^{i+1}\end{aligned}\]
\end{pr}
\begin{proof} Each ordinary point corresponds to a summand of precisely one canonical $H_{a_1^{m_1}\ldots a_s^{m_s}|0}$ as above. Therefore if we add up the number of all the ordinary points on various canonical $H_{a_1^{m_1}\ldots a_s^{m_s}|0}$, we get the number of all ordinary points. Using the inclusion-exclusion principle, it is easy to se that the number of all ordinary points is 
$\sum_{i=0}^{m-1} (-1)^{i+m-1} \binom{m}{i}p^{i+1}$.
\end{proof}

If $r>1$, then the canonical $(m+1)$-tuple $(a_1^{m_1}a_2^{m_2}\ldots a_s^{m_s}|b)$, defined by 
$a_1=0<a_2<\ldots <a_s<p$, $m_1, \ldots, m_s> 0$, $m_1+\ldots+m_s=m$ and $0\leq b<q$ is divisible by $p$,
is obtained by scaling up the canonical element $(a_1^{m_1}a_2^{m_2}\ldots a_s^{m_s}|0)$ from the case $q=p$ earlier.

Analogously as before, we construct a mesh of line segments corresponding to the canonical element $(a_1^{m_1}a_2^{m_2}\ldots a_s^{m_s}|b)$ by first determining all intersection points, 
and then defining line segments that go through these intersection points. Formulating the details of this construction is left to the reader (Also, see the last section of the paper for an alternative formulation).
Finally, define 
\[H_{a_1^{m_1}\ldots a_s^{m_s}|b}=\sum_{\text{line segments}} \quad \sum_{\text{lattice points}} h_{b_1,\ldots, b_s|c}.\]

\begin{pr}\label{m1}
The basis of $SS_r$ for $GL(m|1)$ consists of special elements $\tilde{h}_{a_1, \ldots, a_m|b}$ together with elements
$H_{a_1^{m_1}\ldots a_s^{m_s}|b}$, where $(a_1^{m_1}a_2^{m_2}\ldots a_s^{m_s}|b)$ is canonical. The dimension of $SS_r$ equals 
$q \binom{q-\frac{q}{p}+m-1}{m}+\frac{q}{p}\binom{p+m-2}{m-1}$.
\end{pr}
\begin{proof}
Let $f\in SS_1$ for $GL(m|1)$.
Using Lemma \ref{induction},  the structure of $SS_1$ for $GL(1|1)$ and a suitable extension by inert variables to $GL(m|1)$, we observe that if the coefficients of one of 
$h_{b_1, \ldots, b_s|c}$ that appears as a summand of $H_{a_1^{m_1}\ldots a_s^{m_s}|b}$ is non-zero, then the coefficients of all summands in $H_{a_1^{m_1}\ldots a_s^{m_s}|b}$ are equaled to this nonzero number. Using an induction on the lexicographic order, we can subtract from $f$ a linear combination of $H_{a_1^{m_1}\ldots a_s^{m_s}|b}$ so that the difference will be a linear combination of $h_{b_1, \ldots, b_s|c}$ for special $(b_1, \ldots, b_s|c)$.

To determine the number of special elements $\tilde{h}_{a_1, \ldots, a_m|b}$, we first choose $q$ different values for $b$. For each choice of $b$, 
the number of possible $a_1,\ldots, a_m$ equals the number of combinations with a replacement of $q-\frac{q}{p}$ elements taken $m$ at a time, which equals 
$\binom{q-\frac{q}{p}+m-1}{m}$. Therefore the number of special elements is $q\binom{q-\frac{q}{p}+m-1}{m}$.

As for the number of canonical elements, there are $\frac{q}{p}$ ways to choose $b$. Since $a_1=0$ and $m_1>0$, the number of possible $a_1^{m_1}\ldots a_s^{m_s}$ equals the number of combinations with a replacement of $p$ elements taken $m-1$ at a time, which equals $\binom{p+m-2}{m-1}$. Thus the number of canonical elements equals $\frac{q}{p}\binom{p+m-2}{m-1}$, and the formula for the dimension of $SS_r$ follows.  
\end{proof}

\section{Canonical elements for $GL(m|n)$}

In this section, we define canonical elements for $GL(m|n)$ and compute their cardinality. 
Let $G=GL(m|n)$ and $r>0$. 

\begin{df}
For every $(a_1, \ldots, a_m|b_1, \ldots, b_n)$, where $0\leq a_1, \ldots, a_m;b_1, \ldots, b_n <q$ define
the defect $d$ to be the maximal number $d$ for which  there are index sets $I=\{i_1, \ldots, i_d\}\subset \{1, \ldots, m\}$ and 
$J=\{j_1, \ldots, j_d\}\subset\{1, \ldots, n\}$ such that $a_{i_t}+b_{j_t}\equiv 0 \pmod p$ for $t=1, \ldots, d$. 

We say that $(a_1, \ldots, a_m|b_1, \ldots, b_n)$ of the defect $d=0$ is {\it canonical} if $a_1\leq \ldots \leq a_m$ and $b_1 \leq \ldots \leq b_n$.
In this case $a_i+b_j\not\equiv 0 \pmod p$ for every $1\leq i \leq m$ and $1\leq j\leq n$.

We say that $(a_1, \ldots, a_m|b_1, \ldots, b_n)$ of the defect $d>0$ is {\it canonical} if there are indices $e,f$ such that $d=\min\{e,f\}$ and
\[a_1=\ldots =a_e=0; \,\, 0<a_{e+1}\leq \ldots \leq a_m <p,\]
\[b_1=\ldots =b_{f-1}=0; \,\, b_f \text{ is divisible by } p; \,\, 0<b_{f+1} \leq \ldots \leq b_n<p,\]
and 
\[a_i+b_j\not\equiv 0 \pmod p \text{ for } i>e \text{ and } j>f.\]

Denote by $c_{m,n}(d,q)$ the number of canonical elements of the defect $d$ in $Dist(T_r)$ for $GL(m|n)$. 
\end{df}
If $d>0$ and $q>p$, then $c_{m,n}(d,q)=\frac{q}{p}c_{m,n}(d,p)$. An analogous property is not valid when $d=0$.

Let us recall the well-known formula for the number of ways to choose elements $s_1\leq \ldots \leq s_a$ from a set of cardinality $b$.
This equals the number of combinations with repetitions of $a$ elements selected from the set of cardinality $b$, which is 
$\binom{a+b-1}{a}$. 
Denote by $l$ the cardinality of the set $\{s_1, \ldots, s_a\}$. Then $l$ varies from 1 to $\min\{a,b\}$.
There are $\binom{b}{l}$ ways to choose $l$ distinct entries. After we fix these $l$ distinct entries, 
the remaining entries correspond to permutations of length $a-l$ with repetitions from the set of cardinality $l$, and their number is
$\binom{a-l+l-1}{a-l}=\binom{a-1}{a-l}$.
Therefore we conclude that the number of ways to choose elements $s_1\leq \ldots \leq s_a$ from a set of cardinality $b$ is 
\[\binom{a+b-1}{a}=\sum_{l=1}^{\min\{a,b\}} \binom{b}{l}\binom{a-1}{a-l},\]
where $\binom{b}{l}\binom{a-1}{a-l}$ counts the number of terms $s_1\leq \ldots \leq s_a$ that have $l$ distinct entries.

\begin{df}
Denote by $c_{m,n}(q)$ the number of $(a_1, \ldots, a_m|b_1, \ldots, b_n)$ such that $0\leq a_1\leq \ldots \leq a_m< q$, $0\leq b_1\leq \ldots \leq b_n<q$, 
where $a_i+b_j\not\equiv 0 \pmod p$ for every $i$ and $j$.

Denote by $c'_{m,n}(p-1)$ the number of $(a_1, \ldots, a_m|b_1, \ldots, b_n)$ such that $0< a_1\leq \ldots \leq a_m< p$, $0< b_1\leq \ldots \leq b_n<p$, 
where $a_i+b_j\not\equiv 0 \pmod p$ for every $i$ and $j$.
\end{df}

\begin{lm}\label{c}
If $m,n>0$, then 
\[c_{m,n}(p)=\sum_{l=1}^{\min\{n,p-1\}} \binom{p}{l}\binom{n-1}{n-l}\binom{m+p-l-1}{m}.\]
Also, $c_{m,0}(p)=\binom{m+p-1}{m}$ for $m>0$, $c_{0,n}(p)=\binom{n+p-1}{n}$ for $n>0$, and $c_{0,0}(p)=1$.
If $m,n>0$, then 
\[c'_{m,n}(p-1)=\sum_{l=1}^{\min\{n,p-2\}} \binom{p-1}{l}\binom{n-1}{n-l}\binom{m+p-l-2}{m}.\]
Also, $c'_{m,0}(p-1)=\binom{m+p-2}{m}$ for $m>0$, $c'_{0,n}(p-1)=\binom{n+p-2}{n}$ for $n>0$, and $c'_{0,0}(p-1)=1$.
\end{lm}
\begin{proof}

Assume $m,n>0$. The number of possible entries $0\leq b_1\leq \ldots \leq b_n<p$ is $\sum_{l=1}^{\min\{n,p\}} \binom{p}{l}\binom{n-1}{n-l}$.
For every such entry, where $\{b_1, \ldots, b_n\}$ has the cardinality $l<p$, the entries in $a_1\leq \ldots \leq a_m$ are taken from the set of cardinality 
$p-l$ and their number is $\binom{m+p-l-1}{m}$. Since there are no possible values for $a_1, \ldots, a_m$ if $l=p$, the formula for $c_{m,n}(p)$ follows.

The number of possible entries $0< b_1\leq \ldots \leq b_n<p$ is $\sum_{l=1}^{\min\{n,p-1\}} \binom{p-1}{l}\binom{n-1}{n-l}$.
For every such entry, where $\{b_1, \ldots, b_n\}$ has the cardinality $l<p-1$, the entries in $a_1\leq \ldots \leq a_m$ are taken from the set of cardinality 
$p-1-l$ and their number is $\binom{m+p-l-2}{m}$. Since there are no possible values for $a_1, \ldots, a_m$ if $l=p-1$, the formula for $c'_{m,n}(p-1)$ follows.

The formulae for the cases $m=0$ or $n=0$ are clear.
\end{proof}

Consider permutations of length $a$ with repetitions from the set $S$ of cardinality $c$ divided into different disjoint subsets $S_1, \ldots, S_s$ of cardinalities $b_1, \ldots, b_s$.
The number of all such permutations equals $\binom{a+c-1}{a}$. We will determine the number of such permutations that contains elements of precisely $l$ subsets $S_i$. For each such permutation, denote by $1\leq i_1<\ldots <i_l\leq s$ those indices for which the corresponding number $n_{i_1}, \ldots, n_{i_l}$ of elements from subsets $S_{i_1}, \ldots, S_{i_l}$ is positive
and $n_{i_1}+\ldots +n_{i_l}=a$. For each choice of indices $1\leq i_1<\ldots <i_l\leq s$, and every positive $n_{i_1}, \ldots, n_{i_l}$ such that
$n_{i_1}+\ldots +n_{i_l}=a$, there are $\binom{b_{i_1}+n_{i_1}-1}{n_{i_1}-1}$ ways to choose $n_{i_1}$ elements with repetitions from the set $S_{i_i}$.
Therefore the number of permutations of length $a$ with repetitions from the set $S$ of cardinality $b$ that contains elements of precisely $l$ subsets $S_i$ equals
\[\sum_{1\leq i_1< \ldots < i_l\leq s}\sum_{\substack{n_{i_1}, \ldots, n_{i_l}>0\\n_{i_1}+\ldots +n_{i_l}=a}} \prod_{j=1}^l\binom{b_{i_j}+n_{i_j}-1}{n_{i_j}-1}.\]
If the cardinalities $b_1=\ldots=b_s=b$, then the above formula simplifies to 
\[\binom{s}{l}\sum_{\substack{n_1, \ldots, n_l>0\\n_1+\ldots +n_l=a}} \prod_{j=1}^l\binom{b+n_j-1}{n_j-1}.\]

We will count the number $c_{m,n}(d,q)$ of canonical elements $(a_1, \ldots, a_m|b_1, \ldots, b_n)$ of defect $d$ in $Dist(T_r)$ for $GL(m|n)$.

\begin{lm}\label{c0}
The number $c_{m,n}(0,q)$ of canonical elements of defect zero in $Dist(T_r)$ for $GL(m|n)$ equals 
\[c_{m,n}(q)=  \sum_{l=1}^{\min\{p-1,n\}} \binom{p}{l}\binom{q-\frac{q}{p}l+m-1}{m}\sum_{\substack{n_1, \ldots, n_l>0\\n_1+\ldots +n_l=n}} \prod_{j=1}^l\binom{\frac{q}{p}+n_j-1}{n_j}  \] 
\end{lm}
\begin{proof}
Consider permutations of length $n$ with repetitions from the set $S$ of $\{0, \ldots, q-1\}$ divided into disjoint subsets $S_1, \ldots, S_p$ of cardinalities $\frac{q}{p}$ such that each subset $S_i$ consists of those elements $x$ of $S$ such that $x\equiv i \pmod p$. Then the number of such permutations that contain elements from precisely $l$ subsets $S_i$ equals 
\[\binom{p}{l}\sum_{\substack{n_1, \ldots, n_l>0\\n_1+\ldots +n_l=n}} \prod_{j=1}^l\binom{\frac{q}{p}+n_j-1}{n_j}.\]
This equals the number of ways to choose $0\leq b_1 \leq \ldots \leq b_n<q$ such that residues $b_i \pmod p$ cover precisely $l$ congruence classes modulo $p$.

We need to choose $0\leq a_1 \leq \ldots \leq a_m<q$ in such a way that each $a_i+b_j$ is not divisible by $p$. That means that there are $q-l\frac{q}{p}$ values for each $a_i$ and the number of such $0\leq a_1 \leq \ldots \leq a_m<q$ equals $\binom{q-\frac{q}{p}l+m-1}{m}$. The claim follows. 
\end{proof}

Now assume that $d>0$. Since $c_{m,n}(d,q)=\frac{q}{p}c_{m,n}(d,p)$, it is enough to determine $c_{m,n}(d,p)$, which is the number of 
$(a_1, \ldots, a_m|b_1, \ldots, b_n)$ of defect $d>0$ such that there are $e,f$ for which $d=\min\{e,f\}$,
\[a_1=\ldots =a_e=0; 0<a_{e+1}\leq \ldots \leq a_m <p,\]
\[b_1=\ldots =b_f=0; 0<b_{f+1} \leq \ldots \leq b_n<p,\]
and 
\[a_i+b_j\not\equiv 0 \pmod p \text{ for } i>e \text{ and } j>f.\]

\begin{lm}\label{cand}
If $d>0$, then 
\[c_{m,n}(d,q)=\frac{q}{p}c_{m,n}(d,p)=\frac{q}{p}\sum_{\substack{d\leq e\leq m\\d\leq f\leq n\\\min\{e,f\}=d}} c'_{m-e,n-f}(p-1)   \]
\[=\frac{q}{p}\sum_{d\leq e\leq m} c'_{m-e,n-d}(p-1) + \frac{q}{p}\sum_{d< f\leq n} c'_{m-d,n-f}(p-1).\]
\end{lm}
\begin{proof}
Assume $q=p$ and consider a pair $(e,f)$, where $d\leq e\leq m$, $d\leq f\leq n$ and $\min\{e,f\}=d$.
There are two mutually exclusive cases to consider: either $e\geq d=f$ or $e=d<f$. 
In each case the number of canonical elements corresponding to the pair $(e,f)$
is $c'_{m-e,n-f}(p-1)$.
Therefore we have 
\[c_{m,n}(d,p) = \sum_{d\leq e\leq m} c'_{m-e,n-d}(p-1) + \sum_{d< f\leq n} c'_{m-d,n-f}(p-1).\] 
\end{proof}

Note that the formulae for $c'_{i,j}(p-1)$ were given in Lemma \ref{c}.

\begin{pr}\label{can}
The number $can_{m,n}(q)$ of canonical elements in $Dist(T_r)$ for $GL(m|n)$ is given as
\[can_{m,1}(q)=q\binom{q-\frac{q}{p}+m-1}{m}+\frac{q}{p}\binom{p+m-2}{m-1},\]
\[can_{1,n}(q)=q\binom{q-\frac{q}{p}+n-1}{n}+\frac{q}{p}\binom{p+n-2}{n-1},\] and
\[\begin{aligned} 
can_{m,n}(q)=&\sum_{l=1}^{\min\{p-1,n\}} \binom{p}{l}\binom{q-\frac{q}{p}l+m-1}{m}\sum_{\substack{n_1, \ldots, n_l>0\\n_1+\ldots +n_l=n}} \prod_{j=1}^l\binom{\frac{q}{p}+n_j-1}{n_j}\\
&+\frac{q}{p}\Big[-1+\binom{p+m-2}{m-1} + \binom{p+n-2}{n-1}\Big]\\
&+\frac{q}{p}\sum_{e=1}^{m-1}\sum_{f=1}^{n-1} \sum_{l=1}^{\min\{n-f,p-2\}} \binom{p-1}{l}\binom{f-1}{f-l}\binom{e+p-l-2}{e}\Big]
\end{aligned}\]
for $m,n>1$.
\end{pr}
\begin{proof}
Since the sum $\sum_{d=1}^{\min\{m,n\}} c_{m,n}(d,p)$ equals $\sum_{\substack{1\leq e\leq m\\1\leq f\leq n}} c'_{m-e,n-f}(p-1)$,
Lemmae \ref{c} and \ref{cand} imply that 
\[can_{m,n}(q)=c_{m,n}(q)+\frac{q}{p}\sum_{1\leq e\leq m} \sum_{1\leq f\leq n} c'_{m-e,n-f}(p-1).\]

If $n=1$, then 
\[c_{m,1}(q)=q\binom{q-\frac{q}{p}+m-1}{m}\]
and 
\[\sum_{d\geq 1} c_{m,1}(d,q)= \frac{q}{p}\Big(1 + \sum_{e=1}^{m-1} \binom{m-e+p-2}{m-e}\Big)=\frac{q}{p}\binom{p+m-2}{m-1}\]
which implies the formula for $can_{m,1}(q)$ - see also Proposition \ref{m1}. The case $m=1$ follows by symmetry.

Assume $m,n>1$. Then the sum $\sum_{1\leq e\leq m} \sum_{1\leq f\leq n} c'_{m-e,n-f}(p-1)$ can be broken into four subsums corresponding to the following cases:
1) $e=m, f=n$;
2) $e=m, f<n$;
3) $e<m, f=n$; 
4) $e<m, f<n$,
and expressed as 
\[1 +\sum_{e=1}^{m-1} \binom{m-e+p-2}{m-e} + \sum_{f=1}^{n-1} \binom{n-f+p-2}{n-f} \]
\[+\sum_{e=1}^{m-1}\sum_{f=1}^{n-1} \sum_{l=1}^{\min\{n-f,p-2\}} \binom{p-1}{l}\binom{n-f-1}{n-f-l}\binom{m-e+p-l-2}{m-e}.\]
Using $1 + \sum_{e=1}^{m-1} \binom{m-e+p-2}{m-e}=\binom{p+m-2}{m-1}$ and $1+ \sum_{f=1}^{n-1} \binom{n-f+p-2}{n-f}= \binom{p+n-2}{n-1}$
we can rewrite the last expression as
\[-1 +\binom{p+m-2}{m-1} + \binom{p+n-2}{n-1}
+\sum_{e=1}^{m-1}\sum_{f=1}^{n-1} \sum_{l=1}^{\min\{n-f,p-2\}} \binom{p-1}{l}\binom{f-1}{f-l}\binom{e+p-l-2}{e}.\]
The general formula now follows from Lemma \ref{c0}.
\end{proof}

\section{The basis of $SS_r$ for $GL(m|n)$}

Define the equivalence relation $\sim$ on $(a_1, \ldots, a_m|b_1, \ldots, b_n)$ and on $h_{a_1, \ldots, a_m|b_1, \ldots, b_n}$ generated by the following relations. 

$\bullet \, (a_1, \ldots, a_m|b_1, \ldots, b_n)\sim (a'_1, \ldots, a'_m|b'_1, \ldots, b'_n)$ if \[(a_1, \ldots, a_m|b_1, \ldots, b_n)=\sigma (a'_1, \ldots, a'_m|b'_1, \ldots, b'_n),\] 
where $\sigma\in \Sigma_m\times \Sigma_n$ is a permutation of elements in the first $m$ positions and in the last $n$ positions separately. 

$\bullet \, (a_1, \ldots, a_m|b_1, \ldots, b_n)\sim (a'_1, \ldots, a_m|b'_1, \ldots, b_n)$ whenever $a_1+b_1\equiv 0 \pmod p$ and $a_1+b_1\equiv a'_1+b'_1 \pmod q$.

\begin{pr}
Equivalence classes of $\sim$ correspond uniquely to canonical elements.
\end{pr}
\begin{proof}
Since the defect is an invariant of each equivalence class of $\sim$, we can consider the cases of different defect $d$ separately.

The statement is clear for elements of defect zero. 

Let $(a_1, \ldots, a_m|b_1, \ldots, b_n)$ be of defect $d>0$, and sets $I=\{i_1, \ldots, i_d\}$ and $J=\{j_1, \ldots, j_d\}$ are chosen such that 
$a_{i_t}+b_{j_t}\equiv 0 \pmod p$ for each $t=1, \ldots, d$. The sets $I$ and $J$ are not unique. 
Denote by $R_I$ the collection of residue classes modulo $p$ of the elements $a_i$ for $i\in \{1, \ldots, m\}\setminus I$ and by 
$R_J$ the collection of residue classes modulo $p$ of the elements $b_j$ for $j\in \{1, \ldots, n\}\setminus J$. 
Clearly, $R_I$ does not depend on the choice of $I$, and $R_J$ does not depend on the choice of $J$.

If  $(a_1, \ldots, a_m|b_1, \ldots, b_n)\sim (a'_1, \ldots, a'_m|b'_1, \ldots, b'_n)$, where the pairs of sets $(I,J)$ corresponds to $(a_1, \ldots, a_m|b_1, \ldots, b_n)$
and the pair of sets $(I',J')$ corresponds to $(a'_1, \ldots, a'_m|b'_1, \ldots, b'_n)$, then $R_I=R_{I'}$ and $R_J=R_{J'}$. Additionally, 
\[a_1+\ldots +a_m+b_1+\ldots+b_n\equiv a'_1+\ldots +a'_m+b'_1+\ldots+b'_n  \pmod{q}.\]
Therefore, different canonical elements of defect $d>0$ correspond to different equivalence classes of $\sim$.


Now we show that each $(a_1, \ldots, a_m|b_1, \ldots, b_n)$ of defect $d>0$ is equivalent to a canonical element. We can assume that $a_i+b_i\equiv 0 \pmod p$ for $i=1, \ldots, d$, and then even more, we can assume that $a_i\equiv 0 \pmod p$ and $b_i\equiv 0 \pmod p$ for $i=2, \ldots, d$.

For every $i>1$ write $a_i=r_ip+s_i$, where $0\leq s_i<p$; and for every $j>1$ write $b_j=t_jp+u_j$, where $0\leq u_j<p$. Then 
\[(a_1, \ldots, a_i, \ldots, a_m|b_1, \ldots, b_n)\sim (s_i, a_2, \ldots, r_ip+s_i, \ldots, a_m|q-s_i, b_2, \ldots, b_n)\]
\[\sim (s_i, a_2, \ldots, s_i, \ldots, a_m|q-r_ip-s_i, b_2, \ldots, b_n)=(a'_1, \ldots, s_i, \ldots, a_m|b'_1, b_2, \ldots, b_n),\]
where only the components $a_1, b_1$ and $a_i$ have changed to $a'_1, b'_1$ and $s_i$. We can proceed like this for every $i>1$ and then analogously for every $j>1$ to arrive 
at an element $(a''_1, s_2, \ldots, s_m|b''_1, u_2, \ldots, u_n)$ equivalent to $(a_1, \ldots, a_m|b_1, \ldots, b_n)$. 
Since $s_i=u_i=0$ for $i=2, \ldots, d$ we obtain that 
\[(a''_1,0, \ldots, 0, s_{d+1}, \ldots, s_m|b''_1, 0, \ldots, 0, u_{d+1}, \ldots, u_n) \sim (a_1, \ldots, a_m|b_1, \ldots, b_n).\]
Finally, 
\[(a''_1,0, \ldots, 0, s_{d+1}, \ldots, s_m|b''_1, 0, \ldots, 0, u_{d+1}, \ldots, u_n)\]
\[\sim (0,0, \ldots, 0, s_{d+1}, \ldots, s_m|b'''_1, 0, \ldots, 0, u_{d+1}, \ldots, u_n),\]
where $b'''_1$ is divisible by $p$. 
Once we rearrange the order of the first $m$ entries and the last $n$ entries of the last element, we obtain a canonical element.

 
\end{proof}

\begin{rem}
Note that the proof of the above proposition contains an algorithm for computing the canonical element equivalent to $(a_1, \ldots, a_i, \ldots, a_m|b_1, \ldots, b_n)$.
\end{rem}

Define 
\[H_{a_1, \ldots, a_m|b_1, \ldots, b_n}=\sum\limits_{(a'_1, \ldots, a'_m|b'_1, \ldots, b'_n)\sim (a_1, \ldots, a_m|b_1, \ldots, b_n)} h_{a'_1, \ldots, a'_m|b'_1, \ldots, b'_n}.\]

Each $h_{a'_1, \ldots, a'_m|b'_1, \ldots, b'_n}$ appears as a summand of a unique
$H_{a_1, \ldots, a_m|b_1, \ldots, b_n}$ such that $(a_1, \ldots, a_m|b_1, \ldots, b_n)$ is canonical.

\begin{theo}
The set of elements $H_{a_1, \ldots, a_m|b_1, \ldots, b_n}$, where $(a_1, \ldots, a_m|b_1, \ldots, b_n)$ is canonical, forms a $K$-basis of $SS_r$ for $GL(m|n)$. 
In particular, each $h_{a_1, \ldots, a_m|b_1, \ldots, b_n}$ is a summand of a unique basis element $H_{a_1, \ldots, a_m|b_1, \ldots, b_n}$.

The set of such elements $H_{a_1, \ldots, a_m|b_1, \ldots, b_n}$ is a minimal set of algebra generators of $SS_r$.
\end{theo}
\begin{proof}
Assume that $f\in SS_r$ and the coefficient at $h_{a_1, \ldots, a_m|b_1, \ldots, b_n}$ in $f$ is $c\neq 0$.

If \[(a'_1, \ldots, a'_m|b'_1, \ldots, b'_n)=\sigma (a_1, \ldots, a_m|b_1, \ldots, b_n),\] 
where $\sigma\in \Sigma_m\times \Sigma_n$ is a permutation of elements in the first $m$ positions and those in the last $n$ positions separately, then 
the coefficient at $h_{a'_1, \ldots, a'_m|b'_1, \ldots, b'_n}$ in $f$ equals $c$ because $f$ is symmetric with respect to the action of $\Sigma_m\times \Sigma_n$.

If $(a'_1, \ldots, a_m|b'_1, \ldots, b_n)\sim (a_1, \ldots, a_m|b_1, \ldots, b_n)$, where $a_1+b_1\equiv 0 \pmod p$ and $a_1+b_1\equiv a'_1+b'_1 \pmod q$, 
then since $f$ satisfies the $(\dagger_{1,1})$-condition, Proposition \ref{pr2} and Lemma \ref{induction} imply that the coefficient at $h_{a'_1, \ldots, a_m|b'_1, \ldots, b_n}$ in $f$ also equals $c$.

Therefore, if $(a'_1, \ldots, a'_m|b'_1, \ldots, b'_n)\sim (a_1, \ldots, a_m|b_1, \ldots, b_n)$ in general, then the coefficient at $h_{a'_1, \ldots, a'_m|b'_1, \ldots, b'_n}$ in $f$ equals $c$.

From this, we derive that $f-c H_{a_1, \ldots, a_m|b_1, \ldots, b_n}$ belongs to $SS_r$ but does not contain any summands $h_{a'_1, \ldots, a'_m|b'_1, \ldots, b'_n}$
appearing in $H_{a_1, \ldots, a_m|b_1, \ldots, b_n}$.
Since each $H_{a_1, \ldots, a_m|b_1, \ldots, b_n}\in SS_r$, by induction on the number of nonzero summands $h_{a'_1, \ldots, a'_m|b'_1, \ldots, b'_n}$ in $f$ we conclude that 
$f$ is a linear combination of elements $H_{a_1, \ldots, a_m|b_1, \ldots, b_n}$.

The statement about the minimal set of algebra generators follows from Proposition \ref{torusmn}.
\end{proof}

\begin{rem}
If $(a_1, \ldots, a_m|b_1, \ldots, b_n)$ has defect $d\geq 1$, then the set 
$H_{a_1, \ldots, a_m|b_1, \ldots, b_n}$ can be visualized as a set of specific lattice points lying in the union of $d$-dimensional simplices of $\mathbb{R}^{(m|n)}$ 
in an analogous way it was done in the case $GL(m|1)$ when it was described as a set of lattice points lying in a union of line segments when $d=1$.
\end{rem}


\begin{thebibliography}{99}
\bibitem{chengwang} 
Cheng, S. J. and Wang, W., {\em Dualities and Representations of Lie Superalgebras}, Graduate
Studies in Mathematics, Vol. 144, American Mathematical Society, Providence, RI, 2012.

\bibitem{grmazu} Grishkov, A.N., Marko, Dist(T). and Zubkov, A.N., {\em Generators of supersymmetric polynomials in positive characteristic},
Journal of Algebra 349 (2012) 38-49.

\bibitem{jan} J.~C.~Jantzen,  {\em Representations of algebraic groups}. Second edition. Mathematical Surveys and Monographs, 107. American Mathematical Society, Providence, RI, 2003.

\bibitem{kanttrish} Kantor, Issai and Trishin, Ivan, {\em The algebra of polynomial invariants of the adjoint representation of the Lie superalgebra $gl(m|n)$}, 
Commun. in Algebra, 25(7) (1997), 2039-2070.

\bibitem{lascala} La Scala, R. and Zubkov, A.N., {\em Donkin-Koppinen filtration for general linear supergroup}, Algebr. Represent. Theory 15 (2012), no. 5, 883-899.

\bibitem{marzub} Dist(T).~Marko and A.N.~Zubkov, {\em Blocks for general linear supergroup $GL(m|n)$}, Transformation Groups 23 (1) (2018), 185--215. 

\bibitem{stem} Stembridge, J.R., {\em A characterization of supersymmetric polynomials}, J. Algebra 95(2) (1985), 439-444.

\bibitem{zubmar} Zubkov, A. N. and Marko, Dist(T)., {\em The center of $Dist(GL(m|n))$ in positive characteristic}, Algebra and Representation Theory, 19 (2016), no.3, 613--639.

\end{thebibliography}
\end{document}